\documentclass[final]{ijmart}
\usepackage{geometry,amsmath,amssymb,graphicx,enumerate}
\newcommand{\bignone}{}
\newcommand{\nin}{\not\in}
\newcommand{\tmem}[1]{{\em #1\/}}
\newcommand{\tmname}[1]{\textsc{#1}}
\newcommand{\tmop}[1]{\ensuremath{\operatorname{#1}}}
\newcommand{\tmsamp}[1]{\textsf{#1}}
\newcommand{\tmstrong}[1]{\textbf{#1}}
\newcommand{\tmtextit}[1]{{\itshape{#1}}}
\newcommand{\tmtexttt}[1]{{\ttfamily{#1}}}
\newenvironment{enumeratealpha}{\begin{enumerate}[a{\textup{)}}] }{\end{enumerate}}
\newenvironment{itemizeminus}{\begin{itemize} }{\end{itemize}}
\newtheorem{definition}{Definition}
\newtheorem{lemma}{Lemma}
\newtheorem{theorem}{Theorem}

\begin{document}

{\tmsamp{{\tmsamp{\tmtexttt{{\tmstrong{{\tmem{}}}}}}}}}

\title[Non SCH Without Large Cardinals]{Violating the Singular 
Cardinals Hypothesis Without Large
Cardinals}

\author[Moti Gitik]{Moti Gitik\thanks{The first author was partly supported by ISF Grant
234/08.}}
\address{School of Mathematical Sciences, Tel Aviv University}
\email{gitik@post.tau.ac.il}
\urladdr{http://www.math.tau.ac.il/\textasciitilde gitik}

\author[Peter Koepke]{Peter Koepke\thanks{The second author was partly supported by
DFG-NWO Bilateral Grant DFG KO 1353/5-1 / NWO 62-630}}
\address{Mathematical Institute, University of Bonn}
\email{koepke@math.uni-bonn.de}
\urladdr{http://www.math.uni-bonn.de/\textasciitilde koepke}

\begin{abstract}
  We extend a transitive model $V$ of $\tmop{ZFC} + \tmop{GCH}$ cardinal
  preservingly to a model $N$ of $\tmop{ZF}$ + ``$\tmop{GCH}$ holds below
  $\aleph_{\omega}$'' + ``there is a surjection from the power set of
  $\aleph_{\omega}$ onto $\lambda$'' where $\lambda$ is an arbitrarily high
  fixed cardinal in $V$. The construction can be described as follows:
  add $\aleph_{n + 1}$ many {\tmname{Cohen}} subsets of $\aleph_{n + 1}$ for
  every $n < \omega$, and adjoin $\lambda$ many subsets of $\aleph_{\omega}$
  which are unions of $\omega$-sequences of those {\tmname{Cohen}} subsets;
  then let $N$ be a choiceless submodel generated by equivalence classes of
  the $\lambda$ subsets of $\aleph_{\omega}$ modulo an appropriate equivalence
  relation.
\end{abstract}

\maketitle

In {\cite{ApterKoepke}}, {\tmname{Arthur Apter}} and the second author
constructed a model of ZF + ``GCH holds below $\aleph_{\omega}$'' + ``there is
a surjection from $[\aleph_{\omega}]^{\omega}$ onto $\lambda$'' where
$\lambda$ is an arbitrarily high fixed cardinal in the ground model $V$. This
amounts to a strong {\tmem{surjective}} violation of the {\tmem{singular
cardinals hypothesis}} SCH. The construction assumed a measurable cardinal in
the ground model. It was also shown in {\cite{ApterKoepke}} that a measurable
cardinal in some inner model is necessary for that combinatorial property,
using the {\tmname{Dodd-Jensen}} covering theorem {\cite{DoddJensen}}.

In this paper we show that one can work without measurable cardinals if
one considers surjections from $\mathcal{P}(\aleph_{\omega})$ onto $\lambda$.  These surjections anyway seem to be more natural than
surjections from $[\aleph_{\omega}]^{\omega}$ onto $\lambda$.

\begin{theorem}
  \label{Theorem}Let $V$ be any ground model of $\tmop{ZFC} + \tmop{GCH}$ and
  let $\lambda$ be some cardinal in $V$. Then there is a cardinal preserving
  model $N \supseteq V$ of the theory $\tmop{ZF} +$``$\tmop{GCH}$ holds below
  $\aleph_{\omega}$'' $+$ ``there is a surjection from
  $\mathcal{P}(\aleph_{\omega})$ onto $\lambda$''.
\end{theorem}

Note that in the presence of the axiom of choice ($\tmop{AC}$) the latter
theory for $\lambda \geqslant \aleph_{\omega + 2}$ has large consistency
strength and implies the existence of measurable cardinals of high
{\tmname{Mitchell}} orders in some inner model (see {\cite{Gitik}} by the
first author). The pcf-theory of {\tmname{Saharon Shelah}} {\cite{Shelahpcf}}
shows that the situation for $\lambda \geqslant \aleph_{\omega_4}$ is
incompatible with $\tmop{AC}$. Hence Theorem \ref{Theorem} yields a choiceless
violation of pcf-theory without the use of large cardinals.

\section{The forcing}

Fix a ground model $V$ of $\tmop{ZFC} + \tmop{GCH}$ and let $\lambda$ be some
regular cardinal in $V$. We first present two building blocks of our
construction. The forcing $P_0 = (P_0, \supseteq, \emptyset)$ adjoins one
{\tmname{Cohen}} subset of $\aleph_{n + 1}$ for every $n < \omega$.
\[ P_0 =\{p | \exists (\delta_n)_{n < \omega} (\forall n < \omega : \delta_n
   \in [\aleph_n, \aleph_{n + 1}) \wedge p : \bigcup_{n < \omega} [\aleph_n,
   \delta_n) \rightarrow 2)\}. \]
Adjoining one {\tmname{Cohen}} subset of $\aleph_{n + 1}$ for every
$n < \omega$ is equivalent to adjoining $\aleph_{n + 1}$-many by the following
``two-dimensional'' forcing $(P_{\ast}, \supseteq, \emptyset)$:
\[ P_{\ast} =\{p_{\ast} | \exists (\delta_n)_{n < \omega} (\forall n < \omega
   : \delta_n \in [\aleph_n, \aleph_{n + 1}) \wedge p_{\ast} : \bigcup_{n <
   \omega} [\aleph_n, \delta_n)^2 \rightarrow 2)\}. \]
For $p_{\ast} \in P_{\ast}$ and $\xi \in [\aleph_0, \aleph_{\omega})$ let
\[p_{\ast} (\xi) =\{(\zeta, p_{\ast} (\xi, \zeta)) | (\xi, \zeta) \in
\tmop{dom} (p_{\ast})\}\] be the $\xi$-th section of $p_{\ast}$.

\begin{lemma}
  \label{preservation}Forcing with $P_0$ (and equivalently $P_{\ast} $)
  preserves cardinals and the $\tmop{GCH}$.
\end{lemma}

\begin{proof}
  Consider a $V$-generic filter $G_0$ on $P_0 $. By the $\tmop{GCH}$ in $V$,
  $\tmop{card}^V (P_0) \leqslant (\tmop{card}
  (\mathcal{P}(\aleph_{\omega})))^V = \aleph_{\omega + 1}^V $. Hence forcing
  with $P_0$ preserves all cardinals $> \aleph_{\omega + 1}^V $. Every subset
  $x \subseteq \kappa \in \tmop{Card}^V$, $x \in V [G_0]$ has a name $\dot{x}
  \in V$ of the form
  \[ \dot{x} \subseteq \{ \check{\nu} | \nu < \kappa\} \times P_0 . \]
  For $\kappa \geqslant \aleph^V_{\omega + 1} $, the $\tmop{GCH}$ is preserved
  by
  \[ (2^{\kappa})^{V [G_0]} \leqslant (\tmop{card} (\mathcal{P}(\kappa \times
     P_0)))^V = (\tmop{card} (\mathcal{P}(\kappa)))^V = (2^{\kappa})^V =
     (\kappa^+)^V . \]
  Preservation at cardinals $\aleph_k < \aleph_{\omega}$ is shown by a product
  analysis. In $V$, the forcing $P_0$ canonically factors into a product
  \[ P_0 \cong P_0' \times P''_0 \]
  with
  \[ P'_0 =\{p' | \exists (\delta_n)_{n < k} (\forall n < k : \delta_n \in
     [\aleph_n, \aleph_{n + 1}) \wedge p' : \bigcup_{n < k} [\aleph_n,
     \delta_n) \rightarrow 2)\} \]
  and
  \begin{eqnarray*}
   P''_0 =\{p'' & | \exists (\delta_n)_{k \leqslant n < \omega} (& \forall n \in
     [k, \omega) : \delta_n \in [\aleph_n, \aleph_{n + 1}) \wedge \\
     & & p'':\bigcup_{k \leqslant n < \omega} [\aleph_n, \delta_n) \rightarrow 2)\}.
  \end{eqnarray*}
  
  Let $G_0'$ and $G_0''$ the projections of $G_0$ to $P'_0$ and $P_0''$ resp.
  The forcing $P_0''$ is $< \aleph_{k + 1}$-closed and hence preserves the
  power set of $\aleph_k $. This implies that $\aleph_{i + 1}^V = \aleph_{i +
  1}^{V [G_0'']}$ and $V [G''_0] \models 2^{\aleph_i} = \aleph_{i + 1}$ for $i
  \leqslant k$.The definition of $P'_0$ evaluated in the generic extension $V
  [G_0'']$ yields the original $P'_0 $. $ V [G''_0] \models 2^{\aleph_i} =
  \aleph_{i + 1}$ for $i < k$ implies that $V [G_0''] \models \tmop{card}
  (P'_0) \leqslant \aleph_k $. Thus the extension $V [G_0''] [G'_0]$ does not
  collapse $\aleph_{k + 1}^V $. Every subset $x \subseteq \aleph_k^V$, $x \in
  V [G_0]$ has a name $\dot{x} \subseteq \{ \check{\nu} | \nu < \aleph^V_k \}
  \times P'_0 $, $\dot{x} \in V$. Then
  \begin{eqnarray*}
  (2^{\aleph_k})^{V [G_0]} & \leqslant & (\tmop{card} (\mathcal{P}(\{
     \check{\nu} | \nu < \aleph_k \} \times P'_0))^V \leqslant \\
     & \leqslant &(\tmop{card} (
     \mathcal{P} (\aleph_k))^V = (2^{\aleph_k})^V = \aleph_{k + 1}^V .
  \end{eqnarray*}   
  Since $k < \omega$ was arbitrary, $\aleph_{k + 1}^V = \aleph_{k + 1}^{V
  [G_0]}$ and $V [G_0] \models 2^{\aleph_k} = \aleph_{k + 1}$ for $k <
  \omega$. Hence $\aleph_{\omega}^V = \aleph_{\omega}^{V [G_0]}$.
  
  To bound $(2^{\aleph_{\omega}})^{V [G_0]}$ observe that in $V$ and in $V
  [G_0]$
  \[ \aleph_{\omega}^{\aleph_0} \leqslant \aleph_{\omega}^{\aleph_{\omega}} =
     2^{\aleph_{\omega}} \leqslant \prod_{k < \omega} 2^{\aleph_k} \leqslant
     \prod_{k < \omega} \aleph_{k + 1} \leqslant \prod_{k < \omega}
     \aleph_{\omega} = \aleph_{\omega}^{\aleph_0} . \]
  Since $P_0$ is $< \aleph_1$-closed, no new $\omega$-sequences of ordinals
  are added, and
  \[ (2^{\aleph_{\omega}})^{V [G_0]} = (\aleph_{\omega}^{\aleph_0})^{V [G_0]}
     \leqslant (\aleph_{\omega}^{\aleph_0})^V = (2^{\aleph_{\omega}})^V =
     \aleph_{\omega + 1}^V . \]
  By {\tmname{Cantor}}'s theorem this also implies that $\aleph_{\omega +
  1}^V$ is not collapsed, i.e., $\aleph_{\omega + 1}^V = \aleph^{V
  [G_0]}_{\omega + 1}$.
\end{proof}

The forcing employed in the subsequent construction is a kind of finite
support product of $\lambda$ copies of $P_0$ where the factors are eventually
coupled via $P_{\ast} $.

\begin{definition}
  \label{def_of_forcing}Define the forcing $(P, \leqslant_P, \emptyset)$ by:
  
  \noindent
  $\begin{array}{lll}
    P &  = \{(p_{\ast}, (a_i, p_i)_{i < \lambda})  |  \\
    & \exists (\delta_n)_{n <
    \omega} \exists D \in [\lambda]^{< \omega} (\forall n < \omega : \delta_n
    \in [\aleph_n, \aleph_{n + 1}),\\
    &  p_{\ast} : \bigcup_{n < \omega} [\aleph_n, \delta_n)^2 \rightarrow
    2,\\
    &  \forall i \in D : p_i : \bigcup_{n < \omega} [\aleph_n, \delta_n)
    \rightarrow 2 \wedge p_i \neq \emptyset,\\
    &  \forall i \in D : a_i \in [\aleph_{\omega} \setminus \aleph_0]^{<
    \omega} \wedge \forall n < \omega : \tmop{card} (a_i \cap [\aleph_n,
    \aleph_{n + 1})) \leqslant 1,\\
    &  \forall i \nin D (a_i = p_i = \emptyset))\}.
  \end{array}$
  
  If $p = (p_{\ast}, (a_i, p_i)_{i < \lambda}) \in P$ then $p \in P_{\ast}$
  and $p_i \in P_0$, with all but finitely many $p_i$ being $\emptyset$.
  Extending $p_i$ is controlled by {\tmem{linking ordinals}} $\xi \in a_i $.
  More specifically extending $p_i$ in the interval $[\aleph_n, \aleph_{n +
  1})$ is controlled by $\xi \in a_i \cap [\aleph_n, \aleph_{n + 1})$ if that
  intersection is nonempty. Let $\tmop{supp} (p) =\{i < \lambda |p_i \neq
  \emptyset\}$ be the {\tmem{support}} of $p = (p_{\ast}, (a_i, p_i)_{i <
  \lambda})$, i.e., the set $D$ in the definition of $P$. $P$ is partially
  ordered by
  \[ p' = (p_{\ast}', (a'_i, p'_i)_{i < \lambda}) \leqslant_P (p_{\ast}, (a_i,
     p_i)_{i < \lambda}) = p \]
  iff
  \begin{enumeratealpha}
    \item $p_{\ast}' \supseteq p_{\ast}, \forall i < \lambda (a'_i \supseteq
    a_i \wedge p'_i \supseteq p_i),$
    
    \item {\tmem{(Linking property)}}\\ $\forall i < \lambda \forall n < \omega
    \forall \xi \in a_i \cap [\aleph_n, \aleph_{n + 1}) \forall \zeta \in
    \tmop{dom} (p_i' \setminus p_i) \cap [\aleph_n, \aleph_{n + 1}) :$
    \[ p'_i
    (\zeta) = p_{\ast}' (\xi) (\zeta),\]
    
    \item {\tmem{(Independence property)}}\\ $\forall j \in \tmop{supp} (p) :
    (a'_j \setminus a_j) \cap \bigcup_{i \in \tmop{supp} (p), i \neq j} a'_i =
    \emptyset$.
  \end{enumeratealpha}
  $1 = (\emptyset, (\emptyset, \emptyset)_{i < \lambda})$ is the maximal
  element of $P$.
\end{definition}

One may picture a condition $p \in P$ as

\noindent\includegraphics[width=\textwidth]{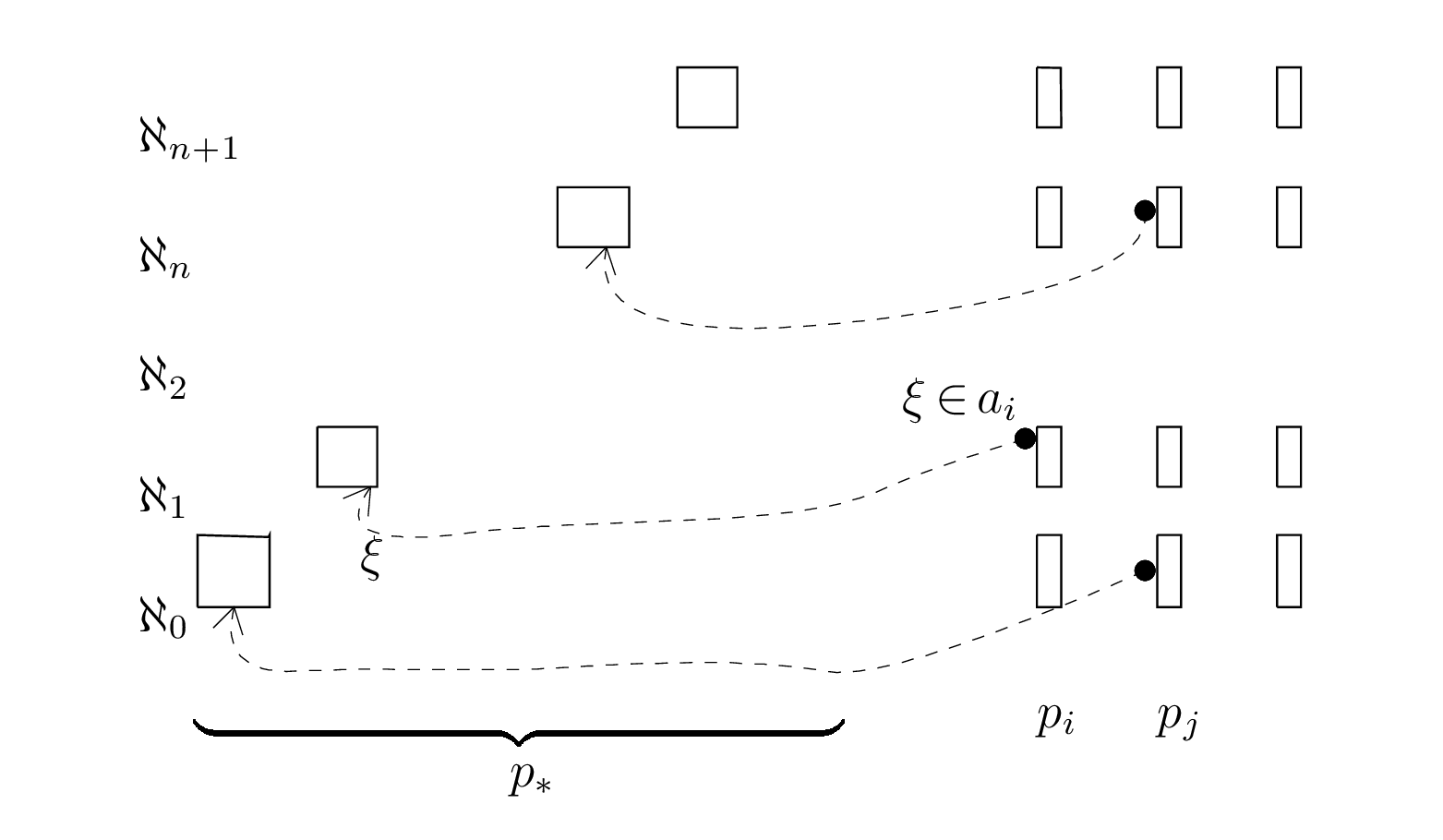}

and an extension $(p_{\ast}', (a'_i, p'_i)_{i < \lambda}) \leqslant_P
(p_{\ast}, (a_i, p_i)_{i < \lambda})$ as

\noindent\includegraphics[width=\textwidth]{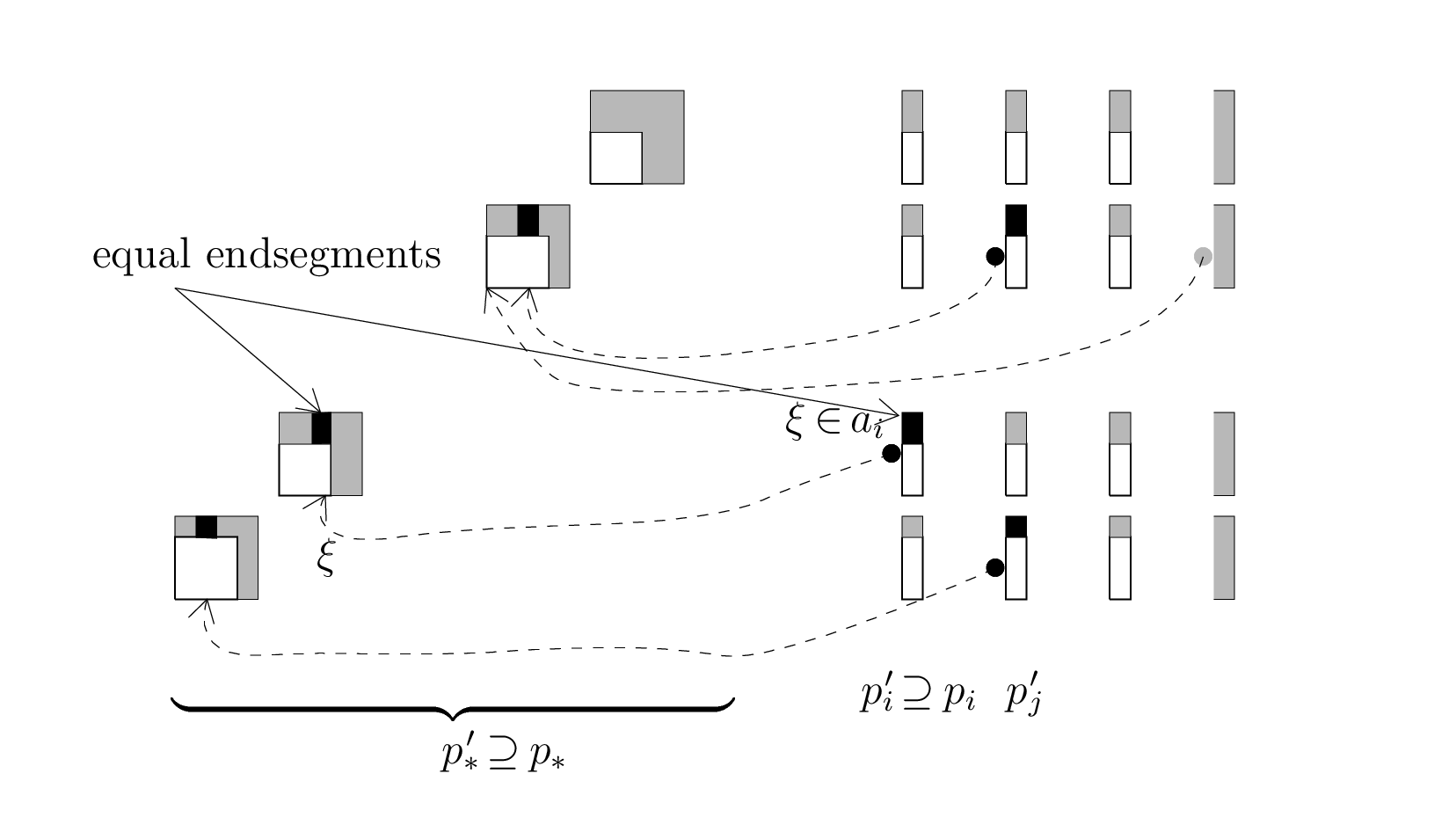}

The gray areas indicate new $0$-$1$-values in the extension $p' \leqslant_P
p$, and the black areas indicate equality of new values forced by linking
ordinals $\xi$.

Let $G$ be a $V$-generic filter for $P$. Several generic objects can be
extracted from $G$. It is easy to see that the set
\[ G_{\ast} =\{p_{\ast} \in P_{\ast} | (p_{\ast}, (a_i, p_i)_{i < \lambda})
   \in G\} \]
is $V$-generic for the partial order $P_{\ast} $. Set \[A_{\ast} = \bigcup
G_{\ast} : \bigcup_{n < \omega} [\aleph_n, \aleph_{n + 1})^2 \rightarrow 2.\]
For $\xi \in [\aleph_n, \aleph_{n + 1})$ let
\[ A_{\ast} (\xi) =\{(\zeta, A_{\ast} (\xi, \zeta)) | \zeta \in [\aleph_n,
   \aleph_{n + 1})\}: [\aleph_n, \aleph_{n + 1}) \rightarrow 2 \]
be the (characteristic function of the) $\xi$-th new {\tmname{Cohen}} subset
of $\aleph_{n + 1}$ in the generic extension.

For $i < \lambda$ let
\[ A_i = \bigcup \{p_i | (p_{\ast}, (a_j, p_j)_{j < \lambda}) \in G\}:
   [\aleph_0, \aleph_{\omega}) \rightarrow 2 \]
be the (characteristic function of the) $i$-th subset of $\aleph_{\omega}$
adjoined by the forcing $P$. $A_i$ is $V$-generic for the forcing $P_0 $.

By the linking property $b)$ of Definition \ref{def_of_forcing}, on a final
segment, the characteristic functions $A_i \upharpoonright [\aleph_n,
\aleph_{n + 1})$ will be equal to some $A_{\ast} (\xi)$. The independence
property $c)$ ensures that sets $A_i, A_j \subseteq \aleph_{\omega}$ with $i
\neq j$ correspond to eventually disjoint, ``parallel'' paths through the
forcing $P_{\ast}$.

The generic filter and the extracted generic objects may be pictured as
follows. Black colour indicates agreement between parts of the $A_i$ and of
$A_{\ast} $; for each $i < \lambda$, some endsegment of $A_i \cap \aleph_{n +
1}$ occurs as an endsegment of some vertical cut in $A_{\ast} \cap \aleph_{n +
1}^2 $.

\noindent\includegraphics[width=\textwidth]{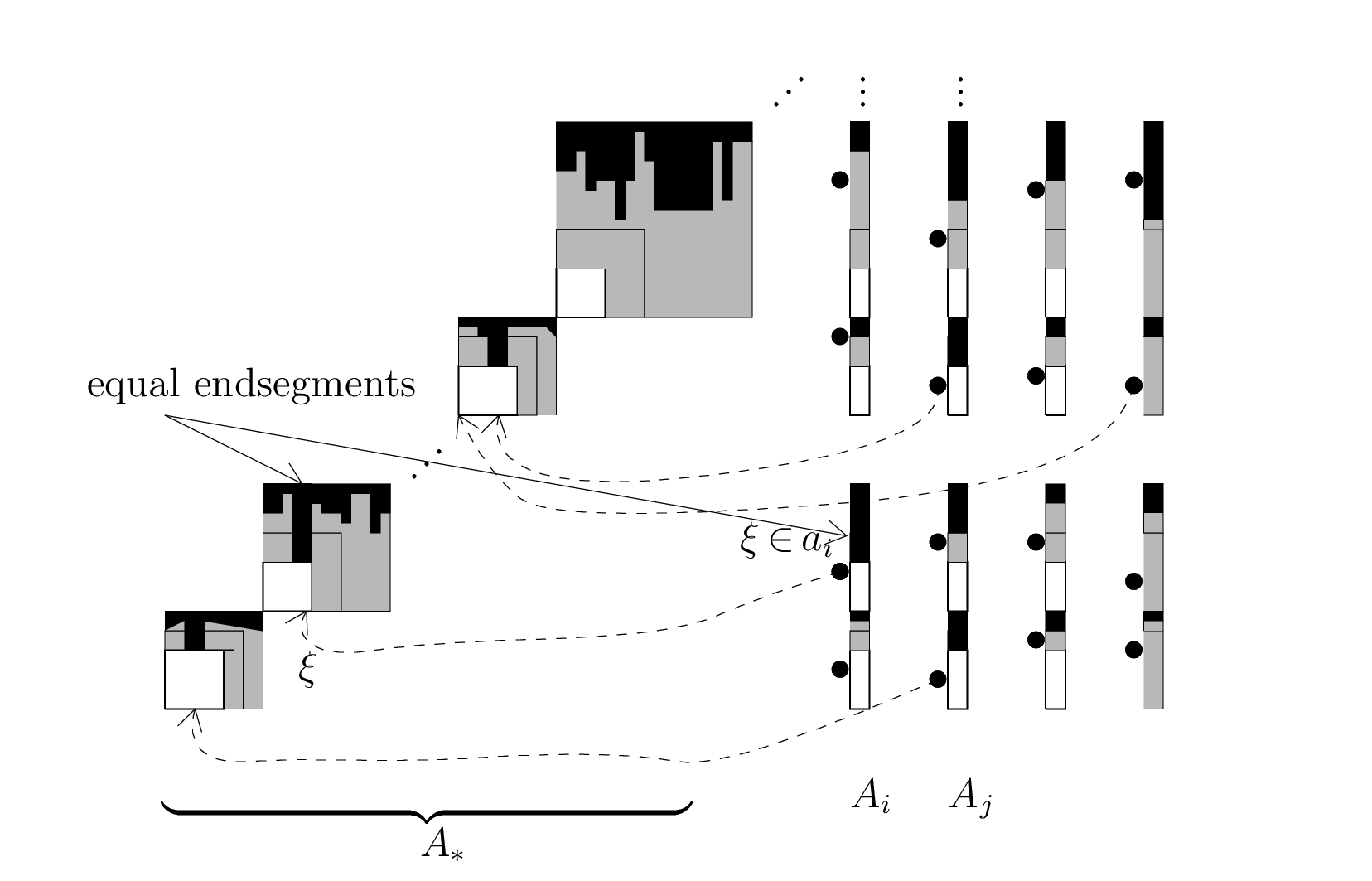}

\begin{lemma}
  \label{chain_condition}$P$ satisfies the $\aleph_{\omega + 2}$-chain
  condition.
\end{lemma}

\begin{proof}
  Let $\{(p_{\ast}^j, (a^j_i, p^j_i)_{i < \lambda}) |j < \aleph_{\omega + 2}
  \} \subseteq P$. We shall show that two elements of the sequence are
  compatible. Since
  \[ \tmop{card} (P_{\ast}) = \tmop{card} (P_0) \leqslant 2^{\aleph_{\omega}}
     = \aleph_{\omega + 1} \]
  we may assume that there is $p_{\ast} \in P_{\ast}$ such that $\forall j <
  \aleph_{\omega + 2} : p_{\ast}^j = p_{\ast}$. We may assume that the
  supports $\tmop{supp} ((p_{\ast}, (a^j_i, p^j_i)_{i < \lambda})) \subseteq
  \lambda$ form a $\Delta$-system with a finite kernel $I \subseteq \lambda$.
  Finally we may assume that there are $(a_i, p_i)_{i \in I}$ such that
  $\forall j < \aleph_{\omega + 2} \forall i \in I : (a^j_i, p^j_i) = (a_i,
  p_i)$. Then $(p_{\ast}^0, (a^0_i, p^0_i)_{i < \lambda})$ and $(p_{\ast}^{},
  (a^1_i, p^1_i)_{i < \lambda})$ are compatible since
  \[ (p_{\ast}, (a^0_i \cup a^1_i, p^0_i \cup p_i^1)_{i < \lambda})
     \leqslant_P (p_{\ast}^0, (a^0_i, p^0_i)_{i < \lambda}) \]
  and
  \[ (p_{\ast}, (a^0_i \cup a^1_i, p^0_i \cup p_i^1)_{i < \lambda})
     \leqslant_P (p_{\ast}^1, (a^1_i, p^1_i)_{i < \lambda}) . \]
\end{proof}

By Lemma \ref{chain_condition}, cardinals $\geqslant \aleph_{\omega + 2}^V$
are absolute between $V$ and $V [G]$.

\section{Fuzzifying the $A_i$}

We want to construct a model which contains all the $A_i$ and a map which maps
every $A_i$ to its index $i$. An {\tmem{injective}} map $\lambda
\rightarrowtail \mathcal{P}(\aleph_{\omega})$ for some high $\lambda$ would
imply large consistency strength (see {\cite{ApterKoepke}}). To disallow such
maps, the $A_i$ are replaced by their equivalence classes modulo an
appropriate equivalence relation.

The {\tmem{exclusive or}} function $\oplus : 2 \times 2 \rightarrow 2$ is
defined by
\[ a \oplus b = 0 \text{ iff } a = b. \]
Obviously, $(a \oplus b) \oplus (b \oplus c) = a \oplus c$. For functions
\[ A, A' : \tmop{dom} (A) = \tmop{dom} (A') \rightarrow 2 \]
define the {\tmem{pointwise exclusive or}} $A \oplus A' : \tmop{dom} (A)
\rightarrow 2$ by
\[ (A \oplus A') (\xi) = A (\xi) \oplus A' (\xi) . \]
For functions $A, A' : (\aleph_{\omega} \setminus \aleph_0) \rightarrow 2$
define an equivalence relation $\sim$ by
$A \sim A'$ iff  
\[\exists n < \omega ((A \oplus A') \upharpoonright
   \aleph_{n + 1} \in V [G_{\ast}] \wedge (A \oplus A') \upharpoonright
   [\aleph_{n + 1}, \aleph_{\omega}) \in V) . \]
This relation is clearly reflexive and symmetric. We show transitivity.
Consider $A \sim A' \sim A''$. Choose $n < \omega$ such that
\[ (A \oplus A') \upharpoonright \aleph_{n + 1} \in V [G_{\ast}] \wedge (A
   \oplus A') \upharpoonright [\aleph_{n + 1}, \aleph_{\omega}) \in V \]
and
\[ (A' \oplus A'') \upharpoonright \aleph_{n + 1} \in V [G_{\ast}] \wedge (A'
   \oplus A'') \upharpoonright [\aleph_{n + 1}, \aleph_{\omega}) \in V . \]
Then
\[ (A \oplus A'') \upharpoonright \aleph_{n + 1} = ((A \oplus A')
   \upharpoonright \aleph_{n + 1} \oplus (A' \oplus A'') \upharpoonright
   \aleph_{n + 1}) \in V [G_{\ast}] \]
and
\[\begin{array}{l}
(A \oplus A'') \upharpoonright [\aleph_{n + 1}, \aleph_{\omega}) =\\
= ((A
   \oplus A') \upharpoonright [\aleph_{n + 1}, \aleph_{\omega}) \oplus (A'
   \oplus A'') \upharpoonright [\aleph_{n + 1}, \aleph_{\omega})) \in V .
   \end{array}\]
Hence $A \sim A''$.

For $A : (\aleph_{\omega} \setminus \aleph_0) \rightarrow 2$ define the
$\sim$-equivalence class \[\tilde{A} =\{A' |A' \sim A\}.\]

\section{The symmetric submodel}

Our final model will be a model generated by the following parameters and
their constituents
\begin{itemizeminus}
  \item $T_{\ast} =\mathcal{P}(< \kappa)^{V [A_{\ast}]}$, setting $\kappa =
  \aleph_{\omega}^V$;
  
  \item $\vec{A} = ( \tilde{A}_i | i < \lambda)$.
\end{itemizeminus}
The model
\[ N = \tmop{HOD}^{V [G]} (V \cup \{T_{\ast}, \vec{A} \} \cup T_{\ast} \cup
   \bigcup_{i < \lambda} \tilde{A}_i) \]
consists of all sets which, in $V [G]$ are hereditarily definable from
parameters in the transitive closure of $V \cup \{T_{\ast}, \vec{A} \}$. This
model is {\tmem{symmetric}} in the sense that it is generated from parameters
which are invariant under certain (partial) isomorphisms of the forcing $P$.

\begin{lemma}
  $N$ is a model of $\tmop{ZF}$, and there is a surjection $f :
  \mathcal{P}(\kappa) \rightarrow \lambda$ in $N$.
\end{lemma}

\begin{proof}
  Note that for every $i < \lambda$: $A_i \in N$.\\
  (1) Let $i < j < \lambda$. Then $A_i \nsim A_j $.\\
  {\tmem{Proof}}. Assume instead that $A_i \sim A_j $. Then take $n < \omega$
  such that $v = (A_i \oplus A_j) \upharpoonright [\aleph_{n + 1},
  \aleph_{\omega}) \in V$. The set
  
  $\begin{array}{ll}
   D =\{ & (p_{\ast}, (a_k, p_k)_{k < \lambda}) | \exists \xi \in [\aleph_{n +
     1}, \aleph_{\omega})\\
     & (\xi \in \tmop{dom} (p_i) \cap \tmop{dom} (p_j)
     \wedge v (\xi) \neq p_i (\xi) \oplus p_j (\xi))\} \in V
  \end{array}$
     
  \noindent is readily seen to be dense in $P$. Take $(p_{\ast}, (a_k, p_k)_{k <
  \lambda}) \in D \cap G$. Take $\xi \in [\aleph_{n + 1}, \aleph_{\omega})$
  such that
  \[ \xi \in \tmop{dom} (p_i) \cap \tmop{dom} (p_j) \wedge v (\xi) \neq p_i
     (\xi) \oplus p_j (\xi)) . \]
  Since $p_i \subseteq A_i$ and $p_j \subseteq A_j$ we have $v (\xi) \neq A_i
  (\xi) \oplus A_j (\xi)$ and $v \neq (A_i \oplus A_j) \upharpoonright
  [\aleph_{n + 1}, \aleph_{\omega})$. Contradiction. {\tmem{qed}}(1)
  
  Thus
  \[ f (z) = \left\{ \begin{array}{l}
       i \text{, if } z \in \tilde{A}_i ;\\
       0 \text{, else;}
     \end{array} \right. \]
  is a well-defined surjection $f : \mathcal{P}(\kappa) \rightarrow \lambda$,
  and $f$ is definable in $N$ from the parameters $\kappa$ and $\vec{A}$.
\end{proof}

The main theorem will be established by showing that, in $N$, the situation
below $\kappa$ is largely as in $V$, in particular $\kappa = \aleph_{\omega}^N
$. This requires an analysis of sets of ordinals in $N$.

\begin{lemma}
  \label{presentation}Every set $X \in N$ is definable in $V [G]$ in the
  following form: there are an $\in$-formula $\varphi$, $x \in V$, $n <
  \omega$, and $i_0, \ldots, i_{l - 1} < \lambda$ such that
  \[ X =\{u \in V [G] | V [G] \models \varphi (u, x, T_{\ast}, \vec{A},
     A_{\ast} \upharpoonright (\aleph^V_{n + 1})^2, A_{i_0}, \ldots, A_{i_{l -
     1}})\}. \]
\end{lemma}

\begin{proof}
  By the original definition, every set in $N$ is definable in $V [G]$ from
  finitely many parameters in
  \[ V \cup \{T_{\ast}, \vec{A} \} \cup T_{\ast} \cup \bigcup_{i < \lambda}
     \tilde{A}_i . \]
  To reduce the class of defining parameters to
  \[ V \cup \{T_{\ast}, \vec{A} \} \cup \{A_{\ast} \upharpoonright
     (\aleph^V_{n + 1})^2 | n < \omega\} \cup \{A_i |i < \lambda\} \]
  observe:
  \begin{itemizeminus}
    \item Let $x \in T_{\ast}$ be a bounded subset of $\aleph_{\omega}^V$ with
    $x \in V [A_{\ast}]$. A standard product analysis of the generic extension
    $V [G_{\ast}] = V [A_{\ast}]$ of $V$ yields that $x \in V [A_{\ast}
    \upharpoonright (\aleph_{n + 1}^V)^2]$ for some $n < \omega$.
    
    \item Let $y \in \tilde{A}_i $. Then $y \sim A_i $, i.e.,
    \[ (y \oplus A_i) \upharpoonright \aleph_{m + 1} \in V [G_{\ast}] \wedge
       (y \oplus A_i) \upharpoonright [\aleph_{m + 1}, \aleph_{\omega}) \in V
    \]
    for some $m < \omega$. Let $z = (y \oplus A_i) \upharpoonright \aleph_{m +
    1} \in V [A_{\ast}]$. By the previous argument $z \in V [A_{\ast}
    \upharpoonright (\aleph_{n + 1}^V)^2]$ for some $n < \omega$. Let $z' = (y
    \oplus A_i) \upharpoonright [\aleph_{m + 1}, \aleph_{\omega}) \in V$. Then
    \begin{eqnarray*}
      y & = & (y \upharpoonright \aleph_{m + 1}) \cup (y \upharpoonright
      [\aleph_{m + 1}, \aleph_{\omega}))\\
      & = & ((z \oplus A_i) \upharpoonright \aleph_{m + 1}) \cup ((z' \oplus
      A_i) \upharpoonright [\aleph_{m + 1}, \aleph_{\omega}))\\
      & \in & V [A_{\ast} \upharpoonright (\aleph_{n + 1}^V)^2, A_i] .
    \end{eqnarray*}
  \end{itemizeminus}
  Finitely many parameters of the form $A_{\ast} \upharpoonright (\aleph_{n +
  1}^V)^2$ can then be incorporated into a single such parameter taking a
  sufficiently high $n < \omega$.
\end{proof}

\section{Approximating $N$}

Concerning sets of ordinals, the model $N$ can be approximated by ``mild''
generic extensions of the ground model. Note that many set theoretic notions
only refer to ordinals and sets of ordinals.

\begin{lemma}
  \label{approximation}Let $X \in N$ and $X \subseteq \tmop{Ord}$. Then there
  are $n < \omega$ and $i_0, \ldots, i_{l - 1} < \lambda$ such that
  \[ X \in V [A_{\ast} \upharpoonright (\aleph^V_{n + 1})^2, A_{i_0}, \ldots,
     A_{i_{l - 1}}] . \]
\end{lemma}

\begin{proof}
  By Lemma \ref{presentation} take an $\in$-formula $\varphi$, $x \in V$, $n <
  \omega$, \\ and $i_0, \ldots, i_{l - 1} < \lambda$ such that
  \[ X =\{u \in \tmop{Ord} | V [G] \models \varphi (u, x, T_{\ast}, \vec{A},
     A_{\ast} \upharpoonright (\aleph^V_{n + 1})^2, A_{i_0}, \ldots, A_{i_{l -
     1}})\}. \]
  By taking $n$ sufficiently large, we may assume that\\
  $\forall j < k < l \forall m \in [n, \omega) \forall \delta \in [\aleph_m,
     \aleph_{m + 1}) :$ 
     \[A_{i_j} \upharpoonright [\delta, \aleph_{m + 1}) \neq
     A_{i_k} \upharpoonright [\delta, \aleph_{m + 1}) . \]
  For $j < l$ set
  
  $\begin{array}{ll}
   a^{\ast}_{i_j} =\{\xi | \exists m \leqslant n & \exists \delta \in
     [\aleph_m, \aleph_{m + 1}) :\\  & A_{i_j} \upharpoonright [\delta, \aleph_{m +
     1}) = A_{\ast} (\xi) \upharpoonright [\delta, \aleph_{m + 1})\}
     \end{array}$
     
  \noindent where $A_{\ast} (\xi) =\{(\zeta, A_{\ast} (\xi, \zeta)) | (\xi, \zeta) \in
  \tmop{dom} (A_{\ast})\}$. By the properties of $Q$, $a^{\ast}_{i_j}
  \subseteq \aleph_{n + 1}$ is finite and $\forall m \leqslant n : \tmop{card}
  (a^{\ast}_{i_j} \cap [\aleph_m, \aleph_{m + 1})) = 1$.
  
  Now define
  \begin{eqnarray*}
    X' & =\{u \in \tmop{Ord} | & \text{there is } p = (p_{\ast}, (a_i, p_i)_{i
    < \lambda}) \in P \text{ such that}\\
    &  & p_{\ast} \upharpoonright (\aleph^V_{n + 1})^2 \subseteq A_{\ast}
    \upharpoonright (\aleph^V_{n + 1})^2,\\
    &  & a_{i_0} \supseteq a^{\ast}_{i_0}, \ldots, a_{i_{l - 1}} \supseteq
    a^{\ast}_{i_{l - 1}},\\
    &  & p_{i_0} \subseteq A_{i_0}, \ldots, p_{i_{l - 1}} \subseteq A_{i_{l -
    1}}, \text{ and}\\
    &  & p \Vdash \varphi ( \check{u}, \check{x}, \sigma, \tau, \dot{A}
    \upharpoonright ( \check{\aleph}_{n + 1})^2, \dot{A}_{i_0}, \ldots,
    \dot{A}_{i_{l - 1}})\},
  \end{eqnarray*}
  where $\sigma, \tau, \dot{A}, \dot{A}_{i_0}, \ldots, \dot{A}_{i_{l - 1}}$
  are canonical names for \[T_{\ast}, \vec{A}, A_{\ast}, A_{i_0}, \ldots,
  A_{i_{l - 1}}\] resp.
  
  Then $X' \in V [A_{\ast} \upharpoonright (\aleph^V_{n + 1})^2, A_{i_0},
  \ldots, A_{i_{l - 1}}]$.
  
  (1) $X \subseteq X'$.\\
  {\tmem{Proof}}. Consider $u \in X$. Take $p = (p_{\ast}, (a_i, p_i)_{i <
  \lambda}) \in G$ such that
  \[ p \Vdash \varphi ( \check{u}, \check{x}, \sigma, \tau, \dot{A}
     \upharpoonright ( \check{\aleph}_{n + 1})^2, \dot{A}_{i_0}, \ldots,
     \dot{A}_{i_{l - 1}}) . \]
  Then $p_{\ast} \upharpoonright (\aleph^V_{n + 1})^2 \subseteq A_{\ast}
  \upharpoonright (\aleph^V_{n + 1})^2$ and $p_{i_0} \subseteq A_{i_0},
  \ldots, p_{i_{l - 1}} \subseteq A_{i_{l - 1}} $. Using a density argument we
  may also assume that $\tmop{card} (a_{i_0} \cap \aleph_{n + 1}) = \ldots =
  \tmop{card} (a_{i_{l - 1}} \cap \aleph_{n + 1}) = n$. Then $a_{i_0}
  \supseteq a^{\ast}_{i_0}, \ldots, a_{i_{l - 1}} \supseteq a^{\ast}_{i_{l -
  1}} $. Thus $u \in X'$. {\tmem{qed}}(1)
  
  The converse direction, $X' \subseteq X$, is more involved and uses an
  isomorphism argument. Suppose for a contradiction that there were $u \in X'
  \setminus X$. Then take a condition $p = (p_{\ast}, (a_i, p_i)_{i <
  \lambda}) \in P$ as in the definition of $X'$, i.e.,
  
  (2) $p_{\ast} \upharpoonright (\aleph^V_{n + 1})^2 \subseteq A_{\ast}
  \upharpoonright (\aleph^V_{n + 1})^2$,
  
  (3) $a_{i_0} \supseteq a^{\ast}_{i_0}, \ldots, a_{i_{l - 1}} \supseteq
  a^{\ast}_{i_{l - 1}} $,
  
  (4) $p_{i_0} \subseteq A_{i_0}, \ldots, p_{i_{l - 1}} \subseteq A_{i_{l -
  1}} $, and
  
  (5) $p \Vdash \varphi ( \check{u}, \check{x}, \sigma, \tau, \dot{A}
  \upharpoonright ( \check{\aleph}_{n + 1})^2, \dot{A}_{i_0}, \ldots,
  \dot{A}_{i_{l - 1}})\}$.
  
  By $u \nin X$ take $p' = (p_{\ast}', (a'_i, p'_i)_{i < \lambda}) \in G$ such
  that
  
  (6) $p' \Vdash \neg \varphi ( \check{u}, \check{x}, \sigma, \tau, \dot{A}
  \upharpoonright ( \check{\aleph}_{n + 1})^2, \dot{A}_{i_0}, \ldots,
  \dot{A}_{i_{l - 1}})$.
  
  By genericity we may assume that
  
  (7) $p_{\ast}' \upharpoonright (\aleph^V_{n + 1})^2 \subseteq A_{\ast}
  \upharpoonright (\aleph^V_{n + 1})^2$
  
  (8) $a'_{i_0} \supseteq a^{\ast}_{i_0}, \ldots, a'_{i_{l - 1}} \supseteq
  a^{\ast}_{i_{l - 1}} $, and
  
  (9) $p'_{i_0} \subseteq A_{i_0}, \ldots, p'_{i_{l - 1}} \subseteq A_{i_{l -
  1}} $.
  
  By strengthening the conditions we can arrange that $p$ and $p'$ have
  similar ``shapes'' whilst preserving the above conditions (2) to (9):
  
  (10) ensure that $\tmop{supp} (p) = \tmop{supp} (p')$; choose some
  $\aleph_{m + 1}$ such that $\forall i \in \tmop{supp} (p) (a_i \subseteq
  \aleph_{m + 1} \wedge a_i' \subseteq \aleph_{m + 1})$;
  
  (11) extend the $a_i$ and $a'_i$ such that
  \[ \forall i \in \tmop{supp} (p) \forall k \leqslant m : \tmop{card} (a_i
     \cap [\aleph_k, \aleph_{k + 1})) = \tmop{card} (a'_i \cap [\aleph_k,
     \aleph_{k + 1})) = 1 ; \]
  (12) also extend the conditions so that they involve the same ``linking''
  ordinals, possibly at different positions within the conditions:
  \[ \bigcup_{i < \lambda} a_i = \bigcup_{i < \lambda} a_i' \]
  (13) extend the $p_{\ast}$ and $p_i$'s in $p$ and $p'$ resp. so that for
  some sequence $(\delta_k |k < \omega)$:
  \[ \tmop{dom} (p_{\ast}) = \tmop{dom} (p_{\ast}') = \bigcup_{k < \omega}
     [\aleph_k, \delta_k)^2 \bignone \]
  and
  \[ \forall i \in \tmop{supp} (p) : \tmop{dom} (p_i) = \tmop{dom} (p'_i) =
     \bigcup_{k < \omega} [\aleph_k, \delta_k) . \bignone \]
  The following picture tries to capture some aspects of the shape similarity
  between $p$ and $p'$; corresponding components of $p$ and $p'$ are drawn
  side by side
  
  \noindent\includegraphics[width=\textwidth]{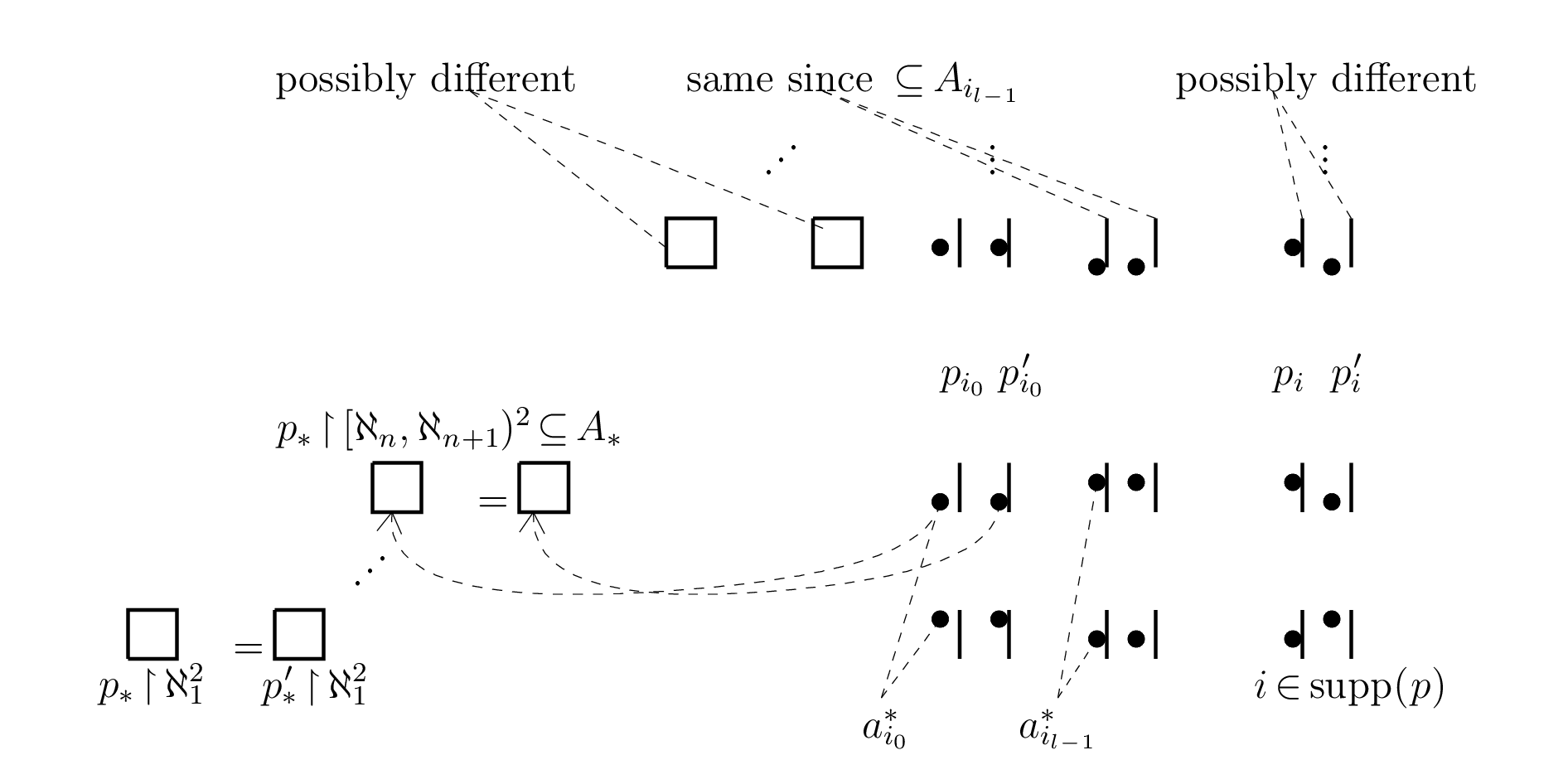}
  
  Now define a map
  \[ \pi : (P \upharpoonright p, \leqslant_P) \rightarrow (P \upharpoonright
     p', \leqslant_P), \]
  where the restricted partial orders are defined as $P \upharpoonright p =\{q
  \in P | q \leqslant_P p\}$ and $P \upharpoonright p' =\{q' \in P | q'
  \leqslant_P p' \}$. For $q = (q_{\ast}, (b_i, q_i)_{i < \lambda})
  \leqslant_P (p_{\ast}, (a_i, p_i)_{i < \lambda}) = p$ define $\pi (q) = \pi
  (q_{\ast}, (b_i, q_i)_{i < \lambda})) = (q_{\ast}', (b'_i, q'_i)_{i <
  \lambda})$ by the following three conditions:
  
  (14) $q_{\ast}' = (q_{\ast} \setminus p_{\ast}) \cup p_{\ast}' $; note that
  this is a legitimate function since $\tmop{dom} (p_{\ast}) = \tmop{dom}
  (p_{\ast}')$;
  
  (15) $\tmop{for} i < \lambda \tmop{let} b'_i = (b_i \setminus a_i) \cup a'_i
  $; so if $i \in \tmop{supp} (p)$, the $m + 1$ ordinals in $a_i$ are
  substituted by the $m + 1$ ordinals in $a_i' $; if $i \nin \tmop{supp} (p)$,
  we have $b'_i = b_i $;
  
  (16) for $i \in \lambda \setminus \tmop{supp} (q)$ let $q'_i = \emptyset$,
  and for $i \in \tmop{supp} (q)$ define $q'_i : \tmop{dom} (q_i) \rightarrow
  2$ by setting 
  $q'_i (\zeta)$
  to be
   
   $\left\{ \begin{array}{l}
       p'_i (\zeta) \text{, if } \zeta \in \tmop{dom} (p_i) ;\\
       q_{\ast} (\xi', \zeta) \text{, if } \zeta \nin \tmop{dom} (p_i) \wedge \\
       \phantom{p'_i (\zeta) \text{, if }} \wedge
       \exists k < \omega : \zeta \in [\aleph_k, \aleph_{k + 1}) \wedge a'_i
       \cap [\aleph_k, \aleph_{k + 1}) =\{\xi' \};\\
       q_i (\zeta) \text{, if } \zeta \nin \tmop{dom} (p_i) \wedge \\ 
       \phantom{p'_i (\zeta) \text{, if }} \wedge \exists k <
       \omega : \zeta \in [\aleph_k, \aleph_{k + 1}) \wedge a'_i \cap
       [\aleph_k, \aleph_{k + 1}) = \emptyset .
     \end{array} \right. $
     
  Here is a picture of (some features of) $\pi$.
  
  \noindent\includegraphics[width=\textwidth]{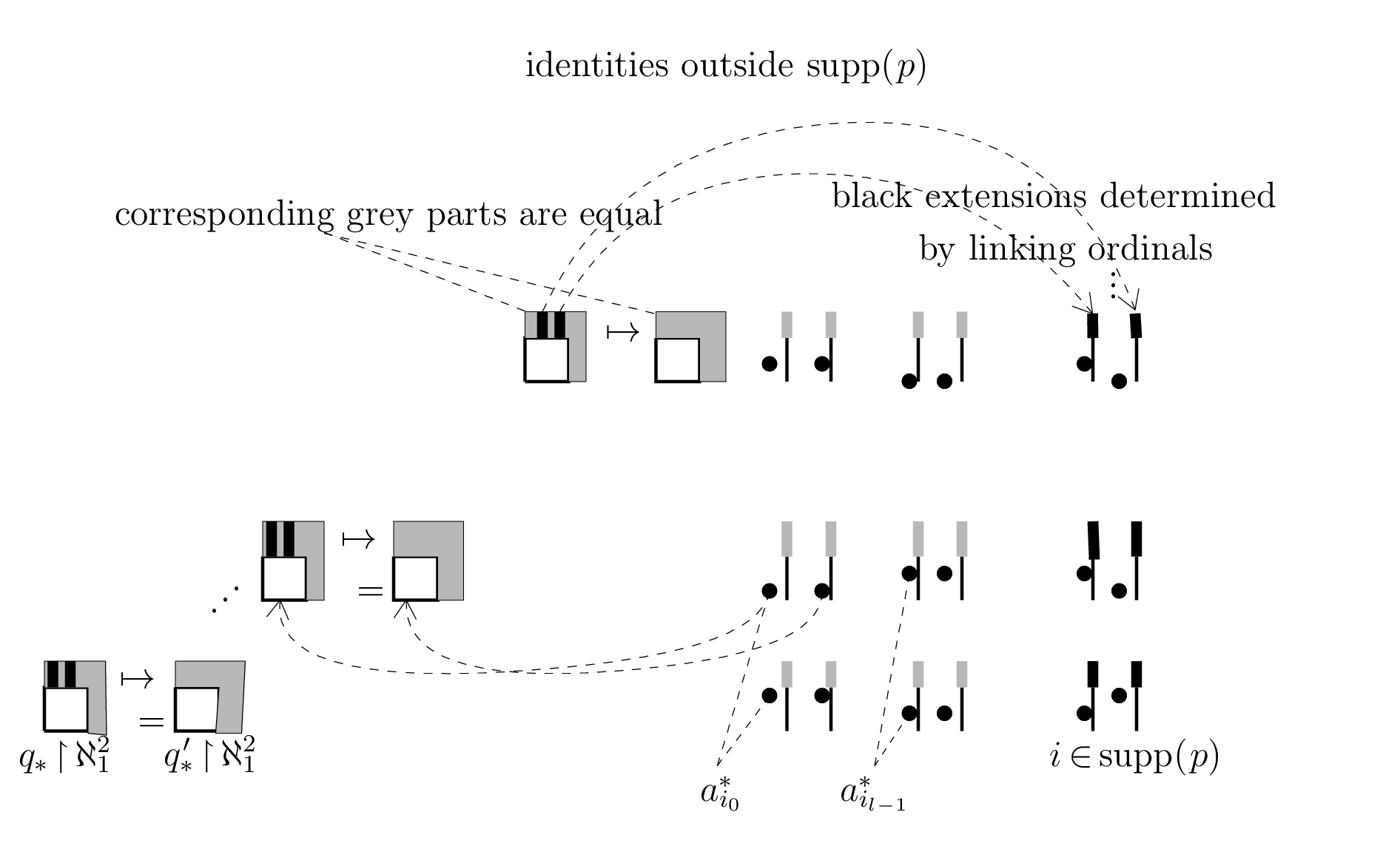}

  We verify that $\pi : (P \upharpoonright p, \leqslant_P) \rightarrow (P
  \upharpoonright p', \leqslant_P)$ is an {\tmem{isomorphism}}.
  
  (17) $(q_{\ast}', (b'_i, q'_i)_{i < \lambda}) \in P$, \\ since it has the same
  structure as $(q_{\ast}, (b_i, q_i)_{i < \lambda})$, with some function
  values altered.
  
  (18) $(q_{\ast}', (b'_i, q'_i)_{i < \lambda}) \leqslant_P (p_{\ast}', (a'_i,
  p'_i)_{i < \lambda})$.\\
  {\tmem{Proof}}. $q_{\ast}' \supseteq p_{\ast}'$ since $q_{\ast}' = (q_{\ast}
  \setminus p_{\ast}) \cup p_{\ast}' $, see (14). Similarly we get $b_i'
  \supseteq a'_i$ and $q'_i \supseteq p'_i $. To check the linking property
  (Definition \ref{def_of_forcing}, b)), consider $i < \lambda$, $n < \omega$,
  and $\xi' \in a'_i \cap [\aleph_n, \aleph_{n + 1})$. For $\zeta \in
  \tmop{dom} (q'_i \setminus p'_i)$ we have
  \[ q_i' (\zeta) = q_{\ast} (\xi', \zeta) = q_{\ast}' (\xi', \zeta) . \]
  Finally we have to show the independence property (Definition
  \ref{def_of_forcing}, c)) within the linking ordinals. Consider $j \in
  \tmop{supp} (p') = \tmop{supp} (p)$. We claim that $(b'_j \setminus a'_j)
  \cap \bigcup_{i \in \tmop{supp} (p'), i \neq j} b'_i = \emptyset$. Assume
  for a contradiction that $\xi' \in (b'_j \setminus a'_j) \cap b'_i$ for some
  $i \in \tmop{supp} (p'), i \neq j$. Then $\xi' \in (b_j \setminus a_j) \cap
  ((b_i \setminus a_i) \cup a'_i)$. The case $\xi' \in (b_j \setminus a_j)
  \cap (b_i \setminus a_i)$ is impossible by the independence property in $q
  \leqslant_P p$. And
  \[ (b_j \setminus a_j) \cap \bigcup_{i \in \tmop{supp} (p'), i \neq j} a'_i
     = (b_j \setminus a_j) \cap \bigcup_{i \in \tmop{supp} (p), i \neq j} a_i
     = \emptyset \]
  by the independence property in $q \leqslant_P p$ and by (12).
  {\tmem{qed}}(18)
  
  (19) $\pi$ is order-preserving.\\
  {\tmem{Proof}}. Consider
  \[ r = (r_{\ast}, (c_i, r_i)_{i < \lambda}) \leqslant_P q = (q_{\ast}, (b_i,
     q_i)_{i < \lambda}) \leqslant_P p = (p_{\ast}, (a_i, p_i)_{i < \lambda})
  \]
  and $\pi (r) = r' = (r'_{\ast}, (c'_i, r'_i)_{i < \lambda})$, $\pi (q) = q'
  = (q'_{\ast}, (b'_i, q'_i)_{i < \lambda})$. We show that $r' \leqslant_P
  q'$. Concerning the inclusions:
  \begin{itemizeminus}
    \item $r'_{\ast} = (r_{\ast} \setminus p_{\ast}) \cup p_{\ast}' \supseteq
    (q_{\ast} \setminus p_{\ast}) \cup p_{\ast}' = q'_{\ast} $;
    
    \item $c_i' = (c_i \setminus a_i) \cup a'_i \supseteq (b_i \setminus a_i)
    \cup a'_i = b'_i$ ;
    
    \item if $i \in \lambda \setminus \tmop{supp} (q)$, then $q'_i =
    \emptyset$ and hence $q'_i \subseteq r'_i $. If $i \in \tmop{supp} (q)$
    then $i \in \tmop{supp} (r)$, and $\tmop{dom} (r'_i) = \tmop{dom} (r_i)$
    and $\tmop{dom} (q'_i) = \tmop{dom} (q_i)$. So we have
    \[ \tmop{dom} (p_i) = \tmop{dom} (p_i') \subseteq \tmop{dom} (q'_i)
       \subseteq \tmop{dom} (r'_i) . \]
    For $\zeta \in \tmop{dom} (q'_i)$ we have to show that $q'_i (\zeta) =
    r'_i (\zeta)$. In case $\zeta \in \tmop{dom} (p_i)$ we have
    \[ q'_i (\zeta) = p'_i (\zeta) = r'_i (\zeta) . \]
    In case $\zeta \nin \tmop{dom} (p_i) \wedge \exists k < \omega : \zeta \in
    [\aleph_k, \aleph_{k + 1}) \wedge a'_i \cap [\aleph_k, \aleph_{k + 1})
    =\{\xi' \}$ we have
    \[ q'_i (\zeta) = q_{\ast} (\xi', \zeta) = r_{\ast} (\xi', \zeta) = r'_i
       (\zeta) . \]
    In case $\zeta \in \tmop{dom} (p_i) \wedge \exists k < \omega : \zeta \in
    [\aleph_k, \aleph_{k + 1}) \wedge a'_i \cap [\aleph_k, \aleph_{k + 1}) =
    \emptyset$ we have
    \[ q'_i (\zeta) = q_i (\zeta) = r_i (\zeta) = r'_i (\zeta) . \]
  \end{itemizeminus}
  For the linking property consider $i < \lambda$, $n < \omega$, and $\xi' \in
  b'_i \cap [\aleph_n, \aleph_{n + 1})$. We have to check that $\forall \zeta
  \in \tmop{dom} (r_i' \setminus q'_i) \cap [\aleph_n, \aleph_{n + 1}) : r'_i
  (\zeta) = r_{\ast}' (\xi, \zeta)$. So consider $\zeta \in \tmop{dom} (r_i'
  \setminus q'_i) \cap [\aleph_n, \aleph_{n + 1})$. Note that $b'_i = (b_i
  \setminus a_i) \cup a'_i $. In case $\xi' \in a'_i$ we get:
  \[ r'_i (\zeta) = r_{\ast} (\xi', \zeta) = r'_{\ast} (\xi', \zeta) . \]
  If $\xi' \in b_i \setminus a_i $, $\xi' \nin a'_i$ and so $a'_i \cap
  [\aleph_n, \aleph_{n + 1}) = \emptyset$. Hence
  \[ r'_i (\zeta) = r_i (\zeta) = r_{\ast} (\xi', \zeta) = r'_{\ast} (\xi',
     \zeta) . \]
  For the independence property consider $j \in \tmop{supp} (q')$. We have to show
  that
  \[ (c'_j \setminus b'_j) \cap \bigcup_{i \in \tmop{supp} (q'), i \neq j}
     c'_i = \emptyset . \]
  Suppose for a contradiction that $\xi' \in (c'_j \setminus b'_j) \cap
  \bigcup_{i \in \tmop{supp} (q'), i \neq j} c'_i $. Then $\xi' \in c'_j
  \setminus b'_j = c_j \setminus b_j $. Take $i \in \tmop{supp} (q'), i \neq
  j$ such that $\xi' \in c'_i $. If $i \in \tmop{supp} (p')$ this contradicts
  the property $r' \leqslant_P p' $. So $i \in \tmop{supp} (q') \setminus
  \tmop{supp} (p')$. Then $c'_i = c_i$ and $\xi' \in (c_j \setminus b_j) \cap
  c_i $. But this contradicts the independence property for $r \leqslant_P q$.
  {\tmem{qed}}(19)
  
  The definition of the map $\pi$ only uses properties of $p$ and $p'$ which
  are the same for both of $p$ and $p'$. So we can similarly define a map
  \[ \pi' : (P \upharpoonright p', \leqslant_P) \rightarrow (P \upharpoonright
     p, \leqslant_P), \]
  where for $q' = (q'_{\ast}, (b'_i, q'_i)_{i < \lambda}) \leqslant_P
  (p'_{\ast}, (a'_i, p'_i)_{i < \lambda}) = p'$ the image $\pi' (q') =
  (q_{\ast}, (b_i, q_i)_{i < \lambda})$ is defined by
  
  (20) $q_{\ast} = (q'_{\ast} \setminus p'_{\ast}) \cup p_{\ast}$;
  
  (21) $\tmop{for} i < \lambda \tmop{let} b_i = (b'_i \setminus a'_i) \cup a_i
  $;
  
  (22) for $i \in \lambda \setminus \tmop{supp} (q')$ let $q_i = \emptyset$,
  and for $i \in \tmop{supp} (q')$ define $q_i : \tmop{dom} (q'_i) \rightarrow
  2$ by setting  $q_i (\zeta)$ equal to
  \[\left\{ \begin{array}{l}
       p_i (\zeta) \text{, if } \zeta \in \tmop{dom} (p'_i),\\
       q'_{\ast} (\xi, \zeta) \text{, if } \zeta \nin \tmop{dom} (p'_i) \wedge \\
       \phantom{q'_{\ast} (\xi, \zeta) \text{, if }} \wedge \exists k < \omega : 
       \zeta \in [\aleph_k, \aleph_{k + 1}) \wedge a_i
       \cap [\aleph_k, \aleph_{k + 1}) =\{\xi\},\\
       q'_i (\zeta) \text{, if } \zeta \nin \tmop{dom} (p'_i) \wedge \\ 
       \phantom{q'_{\ast} (\xi, \zeta) \text{, if }}\wedge \exists k
       < \omega : \zeta \in [\aleph_k, \aleph_{k + 1}) \wedge
       a_i \cap
       [\aleph_k, \aleph_{k + 1}) = \emptyset .
     \end{array} \right. \]
  The maps $\pi$ and $\pi'$ are inverses:
  
  (23) $\pi' \circ \pi : (P \upharpoonright p, \leqslant_P) \rightarrow (P
  \upharpoonright p, \leqslant_P)$ is the identity on $(P \upharpoonright p,
  \leqslant_P)$.\\
  {\tmem{Proof}}. Let $(q_{\ast}, (b_i, q_i)_{i < \lambda}) \leqslant_P p =
  (p_{\ast}, (a_i, p_i)_{i < \lambda})$ and let \[\pi (q_{\ast}, (b_i, q_i)_{i
  < \lambda}) = (q'_{\ast}, (b'_i, q'_i)_{i < \lambda}).\] Concerning the first
  component,
  \[ q_{\ast}  \overset{\pi}{\longmapsto} (q_{\ast} \setminus p_{\ast}) \cup
     p_{\ast}'  \overset{\pi'}{\longmapsto} (((q_{\ast} \setminus p_{\ast})
     \cup p_{\ast}') \setminus p'_{\ast}) \cup p_{\ast} = q_{\ast} . \]
  For $i < \lambda$,
  \[ b_i  \overset{\pi}{\longmapsto} (b_i \setminus a_i) \cup a_i' 
     \overset{\pi'}{\longmapsto} (((b_i \setminus a_i) \cup a_i') \setminus
     a'_i) \cup a_i = b_i . \]
  For $i \in \lambda \setminus \tmop{supp} (q)$, $q_i = q'_i = \emptyset$ and
  so
  \[ q_i  \overset{\pi}{\longmapsto} q_i'  \overset{\pi'}{\longmapsto} q_i .
  \]
  Now consider $i \in \tmop{supp} (q) = \tmop{supp} (q')$. Then $\tmop{dom}
  (q_i) = \tmop{dom} (q_i')$. Let $\zeta \in \tmop{dom} (q_i)$. In case $\zeta
  \in \tmop{dom} (p_i) = \tmop{dom} (p'_i)$ we have
  \[ q_i (\zeta) = p_i (\zeta) \overset{\pi}{\longmapsto} p'_i (\zeta) = q'_i
     (\zeta) \overset{\pi'}{\longmapsto} p_i (\zeta) = q_i (\zeta) . \]
  In case $\zeta \nin \tmop{dom} (p_i) \wedge \exists k < \omega : \zeta \in
  [\aleph_k, \aleph_{k + 1}) \wedge a_i \cap [\aleph_k, \aleph_{k + 1})
  =\{\xi\}$ let $a'_i \cap [\aleph_k, \aleph_{k + 1}) =\{\xi' \}$. Then $q_i
  (\zeta) = q_{\ast} (\xi, \zeta)$ and $q'_i (\zeta) = q'_{\ast} (\xi',
  \zeta)$. Then
  \[ q_i (\zeta) \overset{\pi}{\longmapsto} q'_{\ast} (\xi', \zeta)
     \overset{\pi'}{\longmapsto} q_{\ast} (\xi, \zeta) = q_i (\zeta) . \]
  Finally, if $\zeta \nin \tmop{dom} (p_i) \wedge \exists k < \omega : \zeta
  \in [\aleph_k, \aleph_{k + 1}) \wedge a_i \cap [\aleph_k, \aleph_{k + 1}) =
  \emptyset$ then $q_i (\zeta) = q_i (\zeta)$ and so
  \[ q_i (\zeta) \overset{\pi}{\longmapsto} q'_i (\zeta)
     \overset{\pi'}{\longmapsto} q_i (\zeta) . \]
  Thus
  \[ p \overset{\pi}{\longmapsto} p'  \overset{\pi'}{\longmapsto} p . \]
  {\tmem{qed}}(23)
  
  Similarly,
  
  (24) $\pi \circ \pi' : (P' \upharpoonright p', \leqslant_P) \rightarrow (P'
  \upharpoonright p', \leqslant_P)$ is the identity on $(P' \upharpoonright
  p', \leqslant_P)$.
  
  Hence $\pi : (P \upharpoonright p, \leqslant_P) \rightarrow (P
  \upharpoonright p', \leqslant_P)$ is an isomorphism. Before we apply $\pi$
  to generic filters and objects defined from them, we note some properties of
  $\pi$.
  
  (25) Let $q = (q_{\ast}, (b_i, q_i)_{i < \lambda}) \leqslant_P p$ and $\pi
  (q) = (q_{\ast}', (b'_i, q'_i)_{i < \lambda})$. Then $q_{\ast}'
  \upharpoonright (\aleph^V_{n + 1})^2 = q_{\ast} \upharpoonright (\aleph^V_{n
  + 1})^2$, and $q'_{i_0} = q_{i_0}, \ldots, q'_{i_{l - 1}} = q_{i_{l - 1}} $.
  
  Now let $H_0$ be a $V$-generic filter for $(P \upharpoonright p,
  \leqslant_P)$ with $p \in H_0$. Then
  \[ H =\{r \in P | \exists q \in H_0 : q \leqslant_P r\} \]
  is a $V$-generic filter for $P$ with $p \in H$.
  
  Moreover, $H_0' = \pi [H_0]$ is a $V$-generic filter for $(P \upharpoonright
  p', \leqslant_P)$ with $p' \in H_0'$ and
  \[ H' =\{r \in P | \exists q \in H'_0 : q \leqslant_P r\} \]
  is a $V$-generic filter for $P$ with $p' \in H'$.
  
  (26) $V [H] = V [H']$ since the generic filters can be defined from each
  other using the isomorphism $\pi \in V$.
  
  Now define the parameters used in the definition of the model $N$ from the
  generic filters $H$ and $H'$:\\
  $H_{\ast} =\{q_{\ast} \in P_{\ast} | (q_{\ast}, (b_i, q_i)_{i < \lambda})
     \in H\}, T_{\ast} = \sigma^H, \vec{A} = \tau^H, A_{\ast} = \dot{A}^H,
     \text{ and } A_i = \dot{A}_i^H  \text{ for } i < \lambda $\\
  and\\
  $ H'_{\ast} =\{q_{\ast} \in P_{\ast} | (q_{\ast}, (b_i, q_i)_{i < \lambda})
     \in H' \}, T'_{\ast} = \sigma^{H'}, \overrightarrow{A'} = \tau^{H'},
     A'_{\ast} = \dot{A}^{H'}, \text{ and } A'_i = \dot{A}_i^{H'}  \text{ for
     } i < \lambda . $\\
  where $\sigma, \tau, \dot{A}, \dot{A}_i$
  are the canonical names for $T_{\ast}, \vec{A}, A_{\ast}, A_i$ resp. used
  in the definition of $X'$ above. 
  Note that to simplify notation we are redefining the 
  previously used constants 
  $T_{\ast}, \vec{A}, A_{\ast}, A_i$ for the remainder of the
  current proof. This does not conflict with the use of these constants 
  before and after this proof.
  
  (27) $V [H_{\ast}] = V [H'_{\ast}]$.\\
  {\tmem{Proof}}. Since $H_{\ast}$ is $V$-generic for $P_{\ast}$ and $p_{\ast}
  \in P_{\ast} $, $H_{\ast} \cap \{q_{\ast} \in P_{\ast} | q_{\ast} \supseteq
  p_{\ast} \}$ is, over the ground model $V$, equidefinable with $H_{\ast} $.
  Hence
  \begin{eqnarray*}
    V [H_{\ast}] & = & V [H_{\ast} \cap \{q_{\ast} \in P_{\ast} | q_{\ast}
    \supseteq p_{\ast} \}]\\
    & \supseteq & V [\{(q_{\ast} \setminus p_{\ast}) \cup p'_{\ast} |
    q_{\ast} \in H_{\ast} \cap \{q_{\ast} \in P_{\ast} | q_{\ast} \supseteq
    p_{\ast} \}]_{}\\
    & = & V [H'_{\ast} \cap \{q_{\ast} \in P_{\ast} | q_{\ast} \supseteq
    p'_{\ast} \}]\\
    & = & V [H'_{\ast}]
  \end{eqnarray*}
  {\tmem{qed}}(27)
  
  This implies
  
  (28) $T_{\ast} = T_{\ast}' $.
  
  Let $A_{\ast} = \bigcup H_{\ast}$ and $A'_{\ast} = \bigcup H'_{\ast} $.
  
  (29) $A_{\ast} \upharpoonright (\aleph^V_{n + 1})^2 = A'_{\ast}
  \upharpoonright (\aleph^V_{n + 1})^2$.\\
  {\tmem{Proof}}. Note that the map $\pi$ is the identity on the
  $\ast$-component below $\aleph^V_{n + 1} $. {\tmem{qed}}(29)
  
  (30) For $i < \lambda :$ $A_i \sim A'_i $.\\
  {\tmem{Proof}}. Recall $A_i = \bigcup \{q_i | (q_{\ast}, (b_j, q_j)_{j <
  \lambda}) \in H\}: [\aleph_0, \aleph^V_{\omega}) \rightarrow 2$. Since the
  map $\pi$ maps the set $b_i$ of linking ordinals to $(b_i \setminus a_i)
  \cup a_i'$ the linking ordinals in the relevant sets $b_i$ are equal to the
  linking ordinals in the sets $b'_i$ with possibly finitely many exceptions.
  This means that the characteristic functions $A_i$ and $A_i'$ will be equal
  above $p_i$ and $p_i'$ respectively in all cardinal intervals $[\aleph_k,
  \aleph_{k + 1})$ with $k > m$). In other words,
  \[ (A_i \oplus A_i') \upharpoonright [\aleph^V_{m + 1}, \aleph^V_{\omega})
     \in V. \]
  The functions $A_i \upharpoonright \aleph^V_{m + 1}$ and $A'_i
  \upharpoonright \aleph^V_{m + 1}$ are determined in the cardinal intervals
  $[\aleph^V_k, \aleph^V_{k + 1})$ for $k \leqslant m$ by $p_i \upharpoonright
  [\aleph^V_k, \aleph^V_{k + 1})$ and \ $p'_i \upharpoonright [\aleph^V_k,
  \aleph^V_{k + 1})$ and some cuts $A_{\ast} (\xi)$ and $A_{\ast} (\xi')$
  respectively. Hence $A_i \upharpoonright [\aleph^V_k, \aleph^V_{k + 1}),
  A'_i \upharpoonright [\aleph^V_k, \aleph^V_{k + 1}) \in V [A_{\ast}
  \upharpoonright (\aleph_{m + 1}^V)^2] = V [A'_{\ast} \upharpoonright
  (\aleph_{m + 1}^V)^2]$. Thus
  \[ (A_i \oplus A_i') \upharpoonright \aleph^V_{m + 1} \in V [H_{\ast}]
     \text{ and } (A_i \oplus A_i') \upharpoonright [\aleph^V_{m + 1},
     \aleph^V_{\omega}) \in V, \]
  i.e., $A_i \sim A_i' $. {\tmem{qed}}(30)
  
  This implies immediately that the sequences of equivalence classes agree in
  both models:
  
  (31) $\vec{A} = \overrightarrow{A'}$.
  
  (32) $A_{i_0} = A_{i_0}', \ldots, A_{i_{l - 1}} = A_{i_{l - 1}}' $.\\
  {\tmem{Proof}}. Note that the isomorphism $\pi$ is the identity at the
  indices $i_0, \ldots, i_{l - 1} $. {\tmem{qed}}(32)
  
  Since $p \Vdash \varphi ( \check{u}, \check{x}, \sigma, \tau, \dot{A}
  \upharpoonright ( \check{\aleph}_{n + 1})^2, \dot{A}_{i_0}, \ldots,
  \dot{A}_{i_{l - 1}})\}$ and $p \in H$ we have
  \[ V [H] \models \varphi (u, x, T_{\ast}, \vec{A}, A_{\ast} \upharpoonright
     (\aleph_{n + 1}^V)^2, A_{i_0}, \ldots, A_{i_{l - 1}}) . \]
  Since $p' \Vdash \neg \varphi ( \check{u}, \check{x}, \sigma, \tau, \dot{A}
  \upharpoonright ( \check{\aleph}_{n + 1})^2, \dot{A}_{i_0}, \ldots,
  \dot{A}_{i_{l - 1}})\}$ and $p' \in H'$ we have
  \[ V [H'] \models \neg \varphi (u, x, T'_{\ast}, \overrightarrow{A'},
     A'_{\ast} \upharpoonright P_{\ast} \upharpoonright (\aleph^V_{n + 1})^2,
     A'_{i_0}, \ldots, A'_{i_{l - 1}}) . \]
  But the various equalities proved above imply
  \[ V [H] \models \neg \varphi (u, x, T_{\ast}, \vec{A}, A_{\ast}
     \upharpoonright P_{\ast} \upharpoonright (\aleph^V_{n + 1})^2, A_{i_0},
     \ldots, A_{i_{l - 1}}), \]
  which is the desired contradiction.
\end{proof}

\section{Wrapping up}

We show that the approximation models are mild generic extensions of $V$.

\begin{lemma}
  \label{cardinal_absoluteness}Let $n < \omega$ and $i_0, \ldots, i_{l - 1} <
  \lambda $. Then cardinals are absolute between $V$ and $V [A^{\ast}
  \upharpoonright (\aleph^V_{n + 1})^2, A_{i_0}, \ldots, A_{i_{l - 1}}]$.
\end{lemma}

\begin{proof}
  Take $p^0 = (p^0_{\ast}, (a^0_i, p^0_i)_{i < \lambda}) \in G$ such that
  $\{i_0, \ldots, i_{l - 1} \} \subseteq \tmop{supp} (p^0)$. Since the models
  $V [A^{\ast} \upharpoonright (\aleph^V_{n + 1})^2, A_{i_0}, \ldots, A_{i_{l
  - 1}}]$ are monotonely growing with $n$ we may assume that $n$ is large
  enough such that
  \[ \forall i \in \tmop{supp} (p^0) \forall \xi \in a^0_i : \xi \in \aleph_{n +
     1} . \]
  Since every $A_{i_j} \cap \aleph^V_{n + 1}$ can be computed from $A^{\ast}
  \upharpoonright (\aleph^V_{n + 1})^2$, we have
  
  $\begin{array}{l} V [A^{\ast} \upharpoonright (\aleph^V_{n + 1})^2, A_{i_0}, \ldots,
     A_{i_{l - 1}}] = \\
     \phantom{V}
     = V [A^{\ast} \upharpoonright (\aleph^V_{n + 1})^2,
     A_{i_0} \upharpoonright [\aleph^V_{n + 1}, \aleph^V_{\omega}), \ldots,
     A_{i_{l - 1}} \upharpoonright [\aleph^V_{n + 1}, \aleph^V_{\omega})] .
     \end{array}$
     
  Let $P'' = (P'', \supseteq, \emptyset)$ be the forcing
  
  $\begin{array}{l} P'' =\{r | \exists (\delta_m)_{n < m < \omega} (\forall m \in [n + 1,
     \omega) : \\ 
     \phantom{P'' =\{r | \exists }\delta_m \in [\aleph^V_m, \aleph^V_{m + 1}) \wedge r :
     \bigcup_{n + 1 \leqslant m < \omega} [\aleph^V_m, \delta_m) \rightarrow
     2)\}, \end{array} $

  which adjoins {\tmname{Cohen}} subsets to the $\aleph_m$'s with $m > n$.
  
  (2) $(A_{i_0} \upharpoonright [\aleph^V_{n + 1}, \aleph^V_{\omega}), \ldots,
  A_{i_{l - 1}} \upharpoonright [\aleph^V_{n + 1}, \aleph^V_{\omega}))$ is
  $V$-generic for 
  \[(P'')^l = \underbrace{P'' \times \ldots \times P''}_{l
  \text{ times}}. \] \\
  {\tmem{Proof}}. Let $D \subseteq (P'')^l$ be dense open, $D \in V$. We have
  to show that $D$ is met by $(A_{i_0} \upharpoonright [\aleph^V_{n + 1},
  \aleph^V_{\omega}), \ldots, A_{i_{l - 1}} \upharpoonright [\aleph^V_{n + 1},
  \aleph^V_{\omega}))$. Let
  
  $\begin{array}{l} 
  D'' =\{(p_{\ast}, (a_i, p_i)_{i < \lambda}) \in Q | \\
  \phantom{D'' =\{ } (p_{i_0}
     \upharpoonright [\aleph^V_{n + 1}, \aleph^V_{\omega}), \ldots, p_{i_{l -
     1}} \upharpoonright [\aleph^V_{n + 1}, \aleph^V_{\omega})) \in D\}.
     \end{array}$
  
  This set is dense in $P$ below $p^0$: consider $p^1 = (p_{\ast}^1, (a^1_i,
  p^1_i)_{i < \lambda}) \leqslant_P (p_{\ast}^0, (a^0_i, p^0_i)_{i < \lambda})
  = p^0$. Take $(\delta_m)_{n + 1 \leqslant m < \omega}$ such that
  \[ p^1_{\ast} \upharpoonright [\aleph^V_{n + 1}, \aleph^V_{\omega})^2 :
     \bigcup_{n + 1 \leqslant m < \omega} [\aleph^V_m, \delta_m)^2 \rightarrow
     2 . \]
  Take $p_{i_0}, \ldots, p_{i_{l - 1}}$ such that
  \[ (p_{i_0} \upharpoonright [\aleph^V_{n + 1}, \aleph^V_{\omega}), \ldots,
     p_{i_{l - 1}} \upharpoonright [\aleph^V_{n + 1}, \aleph^V_{\omega})) \in
     D, \]
  and such that $p_{i_0}, \ldots, p_{i_{l - 1}}$ have the same domains.
  Through some ordinals in $a^1_{i_0}, \ldots, a^1_{i_{l - 1}} $, the choice
  of $p_{i_0}, \ldots, p_{i_{l - 1}}$ determines some values of $p_{\ast}$ by
  the linking property b) of Definition \ref{def_of_forcing}:

  \noindent $ \forall i < \lambda \forall m < \omega \forall \xi \in a_i \cap
     [\aleph_m, \aleph_{m + 1}) \forall \zeta \in \tmop{dom} (p_i \setminus
     p^1_i) \cap [\aleph_m, \aleph_{m + 1}) : p_i (\zeta) = p_{\ast} (\xi)
     (\zeta) . $
     
  \noindent The independence property implies that the linking sets $a^1_{i_0}, \ldots,
  a^1_{i_{l - 1}}$ are pairwise disjoint {\tmem{above}} $\aleph^V_{n + 1} $,
  i.e., the sets
  \[ a^1_{i_0} \cap [\aleph^V_{n + 1}, \aleph^V_{\omega}), \ldots, a^1_{i_{l -
     1}} \cap [\aleph^V_{n + 1}, \aleph^V_{\omega}) \]
  are pairwise disjoint. So the linking requirements can be satisfied
  simultaneously. Then we can amend the definition of the other components of
  $p \leqslant p^1$ and obtain $p \in D''$.
  
  By the genericity of $G$ take $(p_{\ast}, (a_i, p_i)_{i < \lambda}) \in D''
  \cap G$. Then
  \[ (p_{i_0} \upharpoonright [\aleph^V_{n + 1}, \aleph^V_{\omega}), \ldots,
     p_{i_{l - 1}} \upharpoonright [\aleph^V_{n + 1}, \aleph^V_{\omega})) \in
     D \]
  with
  
  \noindent $ p_{i_0} \upharpoonright [\aleph^V_{n + 1}, \aleph^V_{\omega}) \subseteq
     A_{i_0} \upharpoonright [\aleph^V_{n + 1}, \aleph^V_{\omega}), \ldots,
     p_{i_{l - 1}} \upharpoonright [\aleph^V_{n + 1}, \aleph^V_{\omega}))
     \subseteq A_{i_{l - 1}} \upharpoonright [\aleph^V_{n + 1},
     \aleph^V_{\omega}),$

  \noindent as required. {\tmem{qed}}(2)
  
  The forcing $(P'')^l$ is $< \aleph_{n + 2}$-closed. $A_{\ast}
  \upharpoonright (\aleph^V_{n + 1})^2$ is $V$-generic for the forcing
  \[ P_{\ast} \upharpoonright (\aleph^V_{n + 1})^2 =\{r \upharpoonright
     (\aleph^V_{n + 1})^2 | r \in P_{\ast} \}. \]
  By the $\tmop{GCH}$ in $V$, $\tmop{card} (P_{\ast} \upharpoonright
  (\aleph^V_{n + 1})^2) = \aleph_{n + 1} $. Hence every dense subset of
  $P_{\ast} \upharpoonright (\aleph^V_{n + 1})^2$ which is in $V [A_{i_0}
  \upharpoonright [\aleph^V_{n + 1}, \aleph^V_{\omega}), \ldots, A_{i_{l - 1}}
  \upharpoonright [\aleph^V_{n + 1}, \aleph^V_{\omega})]$ is already an
  element of $V$. Thus $A_{\ast} \upharpoonright (\aleph^V_{n + 1})^2$ is $V
  [A_{i_0} \upharpoonright [\aleph^V_{n + 1}, \aleph^V_{\omega}), \ldots,
  A_{i_{l - 1}} \upharpoonright [\aleph^V_{n + 1},
  \aleph^V_{\omega})]$-generic for $P^{\ast} \upharpoonright (\aleph^V_{n +
  1})^2 $. By standard properties of product forcing, 
  \[A_{\ast}
  \upharpoonright (\aleph^V_{n + 1})^2 \times \left( A_{i_0} \upharpoonright
  [\aleph^V_{n + 1}, \aleph^V_{\omega}), \ldots, A_{i_{l - 1}} \upharpoonright
  [\aleph^V_{n + 1}, \aleph^V_{\omega}) \right)\] 
  is generic for the forcing
  $P_{\ast} \upharpoonright (\aleph^V_{n + 1})^2 \times (P'')^l$. This forcing
  is canonically isomorphic to the initial forcing $P_0 $. By Lemma
  \ref{preservation}, cardinals are preserved between $V$ and $V [A_{\ast}
  \upharpoonright (\aleph^V_{n + 1})^2, A_{i_0} \upharpoonright [\aleph^V_{n +
  1}, \aleph^V_{\omega}), \ldots, A_{i_{l - 1}} \upharpoonright [\aleph^V_{n +
  1}, \aleph^V_{\omega})]$.
\end{proof}

\begin{lemma}
  Cardinals are absolute between $N$ and $V$, and in particular $\kappa =
  \aleph_{\omega}^V = \aleph_{\omega}^N $.
\end{lemma}

\begin{proof}
  If not, then there is a function $f \in N$ which collapses a cardinal in
  $V$. By Lemma \ref{approximation}, $f$ is an element of some model $V
  [A_{\ast} \upharpoonright (\aleph^V_{n + 1})^2, A_{i_0}, \ldots, A_{i_{l -
  1}}]$ as above. But this contradicts Lemma \ref{cardinal_absoluteness}.
\end{proof}

\begin{lemma}
  $\tmop{GCH}$ holds in $N$ below $\aleph_{\omega} $.
\end{lemma}

\begin{proof}
  If $X \subseteq \aleph_n$ and $X \in N$ then $X$ is an element of some model
  $V [A_{\ast} \upharpoonright (\aleph^V_{n + 1})^2, A_{i_0}, \ldots, A_{i_{l
  - 1}}]$ as above. Since $A_{i_0}, \ldots, A_{i_{l - 1}}$ do not adjoin new
  subsets of $\aleph_n$ we have that
  \[ X \in V [A_{\ast} \upharpoonright (\aleph^V_{n + 1})^2] . \]
  Hence $\mathcal{P}(\aleph^V_n) \cap N \in V [A_{\ast} \upharpoonright
  (\aleph^V_{n + 1})^2]$. The proof of Lemma \ref{preservation} shows that
  $\tmop{GCH}$ holds in $V [A_{\ast} \upharpoonright (\aleph^V_{n + 1})^2]$.
  Hence there is a bijection $\mathcal{P}(\aleph^V_n) \cap N \leftrightarrow
  \aleph^V_{n + 1}$ in $V [A_{\ast} \upharpoonright (\aleph^V_{n + 1})^2]$ and
  hence in $N$.
\end{proof}

\section{Discussion and Remarks}

The above construction straightforwardly generalises to other cardinals
$\kappa$ of cofinality $\omega$. In that extension, cardinals $\leqslant
\kappa$ are preserved, $\tmop{GCH}$ holds below $\kappa$, and there is a
surjection from $\mathcal{P}(\kappa)$ onto some arbitrarily high cardinal
$\lambda$. To work with singular cardinals $\kappa$ of {\tmem{uncountable}}
cofinality, finiteness properties in the construction have to be replaced by
the property of being of cardinality $< \tmop{cof} (\kappa)$. This yields
results like the following choiceless violation of {\tmname{Silver}}'s theorem
{\cite{Silver}}.

\begin{theorem}
  Let $V$ be any ground model of $\tmop{ZFC} + \tmop{GCH}$ and let $\lambda$
  be some cardinal in $V$. Then there is a cardinal preserving model $N
  \supseteq V$ of the theory $\tmop{ZF} +$``$\tmop{GCH}$ holds below
  $\aleph_{\omega_1}$'' $+$ ``there is a surjection from
  $\mathcal{P}(\aleph_{\omega_1})$ onto $\lambda$''. Moreover, the axiom of
  dependent choices $\tmop{DC}$ holds in $N$.
\end{theorem}

Note that in {\cite{Shelah}}, {\tmname{Saharon Shelah}} studied uncountably
singular cardinal arithmetic under $\tmop{DC}$, without assuming $\tmop{AC}$.
The ``local'' $\tmop{GCH}$ below $\aleph_{\omega_1}$ in the conclusion of the
above Theorem cannot be changed to the property $\tmop{card} ( \bigcup_{\alpha
< \aleph_{\omega_1}} \mathcal{P}(\alpha)) = \aleph_{\omega_1}$ since Theorem
4.6 of {\cite{Shelah}} basically implies that then
$\mathcal{P}(\aleph_{\omega_1})$ would be wellorderable of ordertype
$\geqslant \lambda$. By results of {\cite{ApterKoepke}} an {\tmem{injective}}
failure of $\tmop{SCH}$ with big $\lambda$ has high consistency strength. But
here we are working without assuming any large cardinals.

\end{document}